\documentclass[psamsfonts]{amsart}
\usepackage{amsmath,upref,amssymb}
\usepackage{pdfsync}
\RequirePackage{calrsfs}
\usepackage[all]{xy}
\CompileMatrices
\numberwithin{equation}{subsection} 

\theoremstyle{plain}

\newtheorem{theorem}{Theorem}[section]

\newtheorem{lemma}[theorem]{Lemma}

\newtheorem{definition}[theorem]{Definition}

\newtheorem{example}[theorem]{Example}

\theoremstyle{definition}

\theoremstyle{remark}

\begin{document}

\title[On The Weak Order of Orthogonal Groups] {On The Weak Order of Orthogonal Groups }

\author {Annette Pilkington}
 \address{Department of Mathematics\\ University of Notre
Dame\\ Room 255 Hurley Building \\ Notre Dame, Indiana, 46556}
\email{Pilkington.4@nd.edu}


\maketitle    

\begin{abstract}
A structure of a complete lattice (in the sense of a poset)    is defined  on the underlying set of the orhtogonal group  of a real Euclidean space, by a construction  analogous to that of the weak order of   a Coxeter system in terms of its root system.  This gives rise to a complte rootoid
in the sense of Dyer, with the orthogonal group as underlying group.
\end{abstract}

\section{Introduction}

  The weak order of a Coxeter system $(W,S)$ (see Bj{\"o}rner and Brenti \cite{BnB})  is a partial order on the underlying set of $W$  which is important in  the basic combinatorics of $W$. For example, maximal chains in weak order from the identity element to $w\in W$  correspond bijectively to the reduced expressions of $w$. It is known that the weak order of $W$ is a complete meet semi-lattice, which is a complete lattice if $W$ is finite.
 In this paper, we construct  an analogous order associated to the orthogonal group of a real Eucliden space and show that it  is  a complete lattice.

A more  general framework for studying such orders  has  has been developed by  Dyer \cite{Dyer}.
In \cite{Dyer} 
 A signed group-set is defined to be a pair   $(G,\Phi)$ satisfying the following conditions:
 
 \begin{enumerate}\item[(1)]  $G$ is a group
 \item[(2)]  $\Phi$ is a set with  a given left action $(g,x)\mapsto gx$ by the product  group $G\times \{\pm 1\}$. This action defines a pair of commuting left actions, by $G$  and $\{\pm 1\}$ on $\Phi$.
 \item[(3)] There is a specified partition  $\Phi=\Phi^{+}\cup \Phi^{-}$ of $\Phi$ in to two disjoint subsets which are interchanged by the action of $-1$.
 \end{enumerate}
 
 This is, in fact, a special case  of a more general notion of ``signed groupoid-set''   which is not defined  here. We call  $\Phi$ the root system of $(G,\Phi)$  and $\Phi^{+}$ the  set of positive roots. For every $g\in G$, we define the subset 
 $ \Phi_{g}:=\Phi^{+}\cap g(\Phi^{-})$. This is the set of positive roots which are made negative by $g^{-1}$.
 The weak order of  $(G,\Phi)$ is defined to be  the set 
 $\mathcal{L}:=\{\, \Phi_{g}\mid g\in G\,\}$ of subsets of $\Phi^{+}$, partially ordered by inclusion.
 
For a Coxeter system $(W,S)$, the usual standard root system $\Phi$ of $(W,S)$ together with a choice $\Phi^{+}$ of positive roots give rise to a signed group-set $(W,\Phi)$.
The weak-order of this signed group set is exactly the usual weak order of $W$.

In general, for a signed group-set $(G,S)$, the weak order has   good properties if the  additional conditions listed below  are imposed on it. In \cite{Dyer}, $(G,S)$ is said to be rootoidal if the following additional conditions hold:

    \begin{enumerate}\item[(4)] If $g\in {G}$ with $\Phi_{g}=\emptyset$ then $g=1_{G}$.
    \item[(5)]  $\mathcal{L}$ is a complete meet semilattice i.e. each of its non-empty subsets  has a meet (greatest lower boind).
        \item[(6)]  Given  a non-empty family $(A_{i})_{i\in I}$ of  $\mathcal{L}$ and $B\in \mathcal{L}$ such that $A_{i}\cap B=\emptyset$ for all $i\in I$, if  the join  (least upper  bound) $A:=\bigvee_{i\in I} A_{i}$ exists in $\mathcal{L}$, then $A\cap B=\emptyset$. (This condition is called the JTP).
       \end{enumerate}
     
      A particularly interesting class of   rootoidal signed group-sets $(G,\Phi)$ are the 
  complete ones: $(G,\Phi)$ is said to be complete if $\mathcal{L}$ has a maximum element, or equivalently if the weak order of  $\mathcal{L}$ is a complete lattice.             A (complete) signed groupoid-set  gives rise to  an associated (complete)  rootoid, to which the extensive theory  
        developed in   \cite{Dyer} is       applicable. For example, any 
       homomorphism from a group (or groupoid)  $H$ to $G$ gives rise to an associated  
       ``functor rootoid'', which itself may be ``represented'' as a (complete) signed groupoid-set (the use of      groupoid-sets rather than group-sets here is necessary in general  even when $H$ and     $G$ are both groups).

          The main result of this paper is as follows. Let $V$ be a real Euclidean space and 
       $G=O(V)$ be the orthogonal  group of $V$. Let $\Phi$ be the unit sphere in $V$, with its natural action by $G\times\{\pm 1\}$.
       Fix a vector-space total ordering $\preceq$ of $V$ and let $\Phi^{+}:=\{\, v\in \Phi\mid v>0\,\}$ and $\Phi^{-}:=-\Phi^{+}$.
       \begin{theorem} The pair $(G,\Phi)$ defined above is a complete, rootoidal, signed group-set.\end{theorem}

 To prove the theorem,  we shall use the fact that  a signed group-set  $(G,\Phi)$ with weak order $L$ is a complete rootoidal signed group-set if  the following conditions hold:
    \begin{enumerate}\item[(7)] If $g\in {G}$ with $\Phi_{g}=\emptyset$ then $g=1_{a}$.
    \item[(8)]  $L$ is a complete join  semilattice i.e. each of  subsets  has a join 
     \item[(9)] If $A\in \mathcal{L}$, then $\Phi^{+}\setminus A\in \mathcal{L}$.
       \end{enumerate}       
 Indeed, (8) and (9) imply that $L$ is a complete meet semi-lattice by duality.    To prove   that the JTP holds,    let $A_{i}$, $B$ be as in its statement. By (9), $A_{i}\subseteq \Phi^{+}\setminus B\in \mathcal{L}$ for all $i$, so by definition of join, $A=\bigvee A_{i}\subseteq \Phi^{+}\setminus B$ and $A\cap B=\emptyset$.

\noindent
\thanks {\bf Acknowledgement}  \ \ 
I would like to thank Matthew Dyer for some very helpful conversations.  

\section{Notation and Definitions}\label{intro}
If $X$ and $Y$ are sets, we use  $Y \subseteq X$ to denote that $Y$ is a subset  of $X$ or equal to $X$.  The notation $x \in X$ is used to denote that 
$x$ is  an element of $X$. 
If $X$ and $Y$ are sets with $Y \subseteq X$, we use  $X\backslash Y$ to denote the set $\{x \in X | x \not\in Y\}$. 
We use the definitions and notation of \cite{DnP} regarding posets given below. 

\begin{definition}
 A partially ordered set or poset  is a set $P$ equipped with a binary relation $\leq$ such that for all $x, y, $ and $z \in P$,
\begin{enumerate}
\item $x \leq x$,
\item $x \leq y$ and $y \leq x$ imply $x = y$,
\item $x \leq y$ and $y \leq z$ imply $x \leq z$. 
\end{enumerate}
\end{definition}

\begin{definition} A poset $P$ with binary operation $\leq$ is totally ordered if given any  pair of elements $\{x, y \} \in P$ we have either $x \leq y$ or $y \leq x$.  
\end{definition}
A subset of a poset $P$ also has the structure of a poset  which it inherits from $P$. A poset $P$ is a {\bf chain} if for all $x, y \in P$, 
either $x \leq y$ or $y \leq x$.  Two posets $P$ and $Q$ are said to be {\bf order isomorphic}  if there exists a map $\phi$ from $P$ onto $Q$ 
such that $x \leq y$ in $P$ if and only if  $\phi(x) \leq \phi(y)$ in $Q$. Such a map $\phi$ is called an order isomorphism from $P$ to $Q$. 
\begin{definition}
Given a partially  ordered set $P$ with order relation $\leq_{P}$, its dual, $P^{\delta}$, is the poset which is equal to $P$ as a set, with binary relation 
$\leq_{P^{\delta}}$ given by 
$x \leq_{P^{\delta}} y$  if and only if $y \leq_{P} x$. 
\end{definition}

Given a poset $P$ and a subset $S \subseteq P$, an element $x \in P$ is an {\bf upper bound}  for $S$  if $y \leq x$ for all $y \in S$.  An element 
$l \in P$ is a {\bf least upper bound}  for $S$ if $l$ is an upper bound for $S$ and if $x$ is another upper bound for $S$, then $l \leq x$. 
The least upper bound of 
the set $S$ is denoted by $\bigvee S$ when it exists or $\bigvee_P S$ if it is unclear which poset we are considering. In the case of two elements, we let $x \vee y$ denote 
the least upper  bound of the two element set $\{x, y\}$. 
{\bf Lower bounds } and 
{\bf greatest lower bounds}  are defined dually with $\bigwedge S$ denoting the greatest lower bound of the set $S$ and $x \wedge y$ denoting the 
greatest lower bound of the set $\{x, y\}$. 

\begin{definition}
Let $P$ be a non-empty poset. If $x\wedge y$ and $x \vee y$ exist in $P$ for all $x, y \in P$, we say $P$ is a {\bf Lattice}. If $\bigwedge S$ and $\bigvee S$ exist in $P$ 
for every subset  $S$ of $P$, then we say $P$ is a {\bf complete lattice}. 
\end{definition}

If $S$ is a non-empty subset of a poset $P$, $S$ is said to be {\bf directed} if for every pair of elements $x, y \in S$, there exists $z \in S$, such that $z$ is an
upper bound for the set $\{x, y\}$.

We let $\Bbb{R}^{n + 1}$ denote The Euclidean $(n + 1)$ space, with standard basis $\{e_1, e_2, \dots , e_{n + 1}\}$.   The standard  inner product or dot product on $\Bbb{R}^{n+1}$ is denoted by  $< , >$. The unit sphere in 
$\Bbb{R}^{n+1}$ is 
$$S^n = \{(x_1, x_2, \dots , x_{n+1}) \in \Bbb{R}^{n+1}  | x_1^2 + x_2^2 + \dots + x_{n+1}^2 = 1\}$$
 and $\Phi_n^+$
 is used to denote the set 
 $$\Phi_n^+ = \{(x_1, x_2, \dots , x_{n + 1}) \in S^n | x_m  = <x, e_m>  > 0 \ \ \mbox{where} \ \ m = \mbox{max}\{i | x_i \not= 0\} \}.$$
For its negative   $ -\Phi_n^+ $, we use 
$$\Phi_n^- =  \{(x_1, x_2, \dots , x_{n + 1}) \in S^n | x_m  = <x, e_m> < 0 \ \ \mbox{where} \ \ m = \mbox{max}\{i | x_i \not= 0\} \}.$$
Note that $S^n = \Phi_n^+ \cup \Phi_n^-$ and $\Phi_n^+ \cap \Phi_n^-$ is empty. 

For  a Euclidean  space $V$, we let $O(V)$ denote the orthogonal group of $V$. Sometimes we use 
 $O_{k}$  to denote the orthogonal transformations of  $\Bbb{R}^{k}$.
For any $T \in O_{n + 1}$, we have  $<T(x), T(y)> = <x, y>$ for all  $x, y \in \Bbb{R}^{n + 1}$.  
We write $x\bot y$ to denote that  $<x, y> = 0$ for $x, y \in \Bbb{R}^{n + 1}$. 
If $V_1$ and $V_2$ are subspaces of $\Bbb{R}^{n + 1}$ with $V_1 + V_2 = V$ and $<v_1, v_2> = 0$ for all $v_1 \in V_1$ and $v_2 \in V_2$, then we 
say that $V_2$ is the orthogonal complement of $V_1$ in $V$ and we write the sum as $V = V_1 \bot V_2$. If $T_1 \in O(V_1)$ and $T_2 \in O(V_2)$ and $V = V_1 \bot V_2$, the linear  transformation $T = T_1 \bot T_2$ is the unique linear transformation in 
$O(V)$ with the property that  $T(v_1) = T_1(v_1)$ and $T(v_2) = T_2(v_2)$ for all $v_1 \in V_1$ and $v_2 \in V_2$. We let $I_V$ denote the identity transformation on $V$ and we let $1_{e_i}$ denote the identity transformation on the subspace $\Bbb{R}e_i$ of $\Bbb{R}^{n + 1}$. 

We  have a natural imbedding of $\Bbb{R}^{k}, k < n$  in $R^{n + 1}$ where we identify 
$\Bbb{R}^{k}$ with the set $\{(x_1, x_2, \dots , x_{n + 1} )| x_m = 0 \ \mbox{if} \ m > k \}$. 
This induces an imbedding of $S^{k - 1}, \Phi_{k - 1}^+$  and $\Phi_{k - 1}^-$ in $R^{n + 1}$. 
For $x = (x_1, x_2, \dots , x_{n + 1})  \in \Bbb{R}^{n + 1}$, we let $||x|| = \sqrt{x_1^2 + x_2^2 + \dots + x_{n + 1}^2}$. 
The topological closure of a set $X \subseteq \Bbb{R}^{n + 1}$ with respect to the metric defined by this norm on $\Bbb{R}^{n + 1}$ is denoted by $\bar{X}$.

\section{Convex cones }
We will use the terminology   of \cite{Fenchel} with regard to convex cones. This may differ from other texts. In this section we include a series of Lemmas, proving some basic facts for 
which we cannot find direct references in the literature. 

\begin{definition} \label{cones} Let $V$ be a Euclidean space with positive definite inner product $<- , ->$. A subset $C \subseteq V$ is called a {\bf cone} if $C$ satisfies  conditions (1) and (2) below:
\begin{enumerate}
\item $0 \in C$,
\item If $v \in C$, then $\Bbb{R}_{\geq 0}v \subseteq  C$.

\noindent
\leftline{C is called a {\bf convex cone} if it also satisfies the property:}

\item If $u, v \in C$, then $u + v \in C$.

\noindent 
A convex cone $C$  is a  {\bf pointed convex cone}
 if it is a convex cone and satisfies the property:
\item  $C \cap (-C) = \{0\}$.
\end{enumerate}
\end{definition}

Note that for any cone $C$ contained in $R^k$, $k > 0$, if $x \not= 0  \in C$, the {\bf  ray}  $(x) = \{rx | r \in \Bbb{R}_{\geq 0}\}$ is contained in $C$. This ray 
intersects the unit sphere $S^{k - 1}$ at exactly one point $\frac{x}{||x||}$. Thus we have a one to one correspondence between the rays in $C$ and the points on the unit sphere $S^{k - 1}$ which are  in $C$. 

\begin{definition} \label{hatconedef} Let $X \subseteq S^n$, we use  $\widehat{X}$ to denote the following  cone in $\Bbb{R}^{n + 1}$:
$$\widehat{X} = \{rx | x \in X \ \mbox{and} \ r \in \Bbb{R}_{\geq 0} \} \ \mbox{if} \ X \not= \emptyset$$
and $\widehat{\emptyset} = \{0\}$. 

\end{definition}

\begin{lemma} \label{hatcone} Let $X$ and $Y$ be subsets of $S^n$.
\begin{enumerate}
\item $X = \widehat{X} \cap S^n$ and   if 
$C$ is any cone  in $\Bbb{R}^{n + 1}$ with $X = C\cap S^n$, then $\widehat{X} = C$,
\item  $X\cap Y = \emptyset$, if and only if  $\widehat{X} \cap \widehat{Y} = \{0\}$.
\end{enumerate}
\end{lemma}
\begin{proof} Assertion (1) is trivial if $X = \emptyset$, so let us assume for now that  $X \not= \emptyset$. 
We have $\widehat{X}$ is a union of rays,  $\widehat{X} = \bigcup_{x \in X, x \not= 0} (x)$. Since each ray $(x)$ where 
$x \in X$ 
cuts the unit sphere exactly once  at $x$, we have $X = \widehat{X} \cap S^n$. If C is a cone with $C \cap S^n = X$, then 
$\widehat{X} \subseteq C$, since the entire  ray $(x)$ is in $C$ if $x \in X$. On the other hand, if $c \in C$, the ray $(c)$ cuts 
the sphere $S^n$ at $\frac{c}{||c||} \in X$. Therefore $(\frac{c}{||c||}) = (c) \in \widehat{X}$ and since $C$ is the union of such rays, we have $C \subseteq \widehat{X}$. This proves (1).

We can now relax the assumption that $X \not= \emptyset$. 
It is easily seen that  $\widehat{X} \cap \widehat{Y} = \widehat{X \cap Y}$. Assertion (2) follows from the observation that 
if $C$ is a cone in  $\Bbb{R}^{n + 1}$, then $C \cap S^n = \emptyset$ if and only if $C = \{0\}$. 
\end{proof}

\begin{lemma} \label{cone} Let $A$ be a convex cone in $\Bbb{R}^{n + 1}$, then $\overline{A \cap S^{n} } = \bar{A} \cap S^n$.
\end{lemma}
\begin{proof}
Let   $A \subset \Bbb{R}^{n + 1}$ be a convex cone.  Since 
$ \bar{A} \cap S^n$ is closed and contains $A \cap S^{n}$, we have $\overline{A \cap S^{n} } \subseteq  \bar{A} \cap S^n$. On the other hand, if $x \in \bar{A} \cap S^n$, then $||x|| = 1$ and we have a sequence of points $\{x_i\}_{i = 1}^{\infty}$ in $A$ such that
$\lim_{i \to \infty}x_i = x$. Since $A$ is convex we have the sequence of points $\{\frac{x_i}{||x_i||}\}_{i = 1}^{\infty} \in  \bar{A} \cap S^n$. 
Given any $\epsilon > 0$, there exists $m$ such that  $| ||x_i|| - 1 | = | ||x_i|| - ||x|| | \leq   ||x_i - x|| < \epsilon $ for all $i \geq m$.
Since 
$$||x - \frac{x_i}{||x_i||}|| \leq ||x - x_i|| + ||x_i -  \frac{x_i}{||x_i||}|| = ||x - x_i|| + \frac{||x_i||(| \ ||x_i||  -  1|)}{||x_i||} \leq 2\epsilon$$
for $i \geq m$, we have that the sequence $\{\frac{x_i}{||x_i||}\}_{i = 1}^{\infty}$ converges to $x$ and $x \in \overline{A \cap S^{n} }$. 
\end{proof}

\begin{definition} Let $V$ be a Euclidean space with positive definite inner product $<- , ->$. For any fixed $u  \in V$, we let $H_u$ denote {\bf the closed half space}   $H_u = \{v \in V | <v, u>  \leq 0 \}$. 
We let $H(u)$ denote the {\bf hyperplane} $\{ v \in V | 
<v, u> = 0 \}$ and we let $H_u^{int}$ denote the {\bf open half space }
$\{ v \in V | 
<v, u> < 0 \}$. \\
We will use the notation $H^V_u, H^V(u)$, and $H^{Vint}_u$ when we want to emphasize  the 
underlying vector space.\\
\end{definition}

 \begin{lemma} \label{hyper}Let $V$ be a Euclidean space with positive definite inner product \\$<- , ->$. If 
 $u, v \in V$ with $H(u) \subseteq H_v$,
  then   $H(u) = H(v)$. Consequently if  $u, v \in V$ with
  $H_u \subseteq H_v$, then $H_u = H_v$.
  Also  $v_1 = \frac{-v}{||v||}  \in H_v$ is the unique vector of norm one in $H_v$ with the property that $H_{v} = H_{-v_1}$. 
\end{lemma}
\begin{proof} Suppose $u, v \in V$ with  $H(u) \subseteq H_v = \{x \in V | <x, v> \leq 0 \}$. If $x \in H(u)$ with the property that 
$<x, v> \not= 0$, then $\{x,-x\} \in H(u)$ and either  $<x, v>  > 0$ or $<-x, v> > 0$ giving a contradiction. Hence  $H(u) \subseteq
H(v)$ and we have equality by comparing dimensions. 

If $H_u \subseteq H_v$, then $H(u) \subseteq H_v$ and therefore $H(u) = H(v)$. Now  $V  = <v> \bot H(v) = <u> \bot H(u) = <u> \bot H(v)$ and $u = cv$. Since 
$H_u \subseteq H_v$,  $-u \in H_v$, hence
$<-u, v> = <-cv, v> < 0$ and we have $c > 0$. Therefore $H_u = H_v$. 

Let $v_1 \in H_v$ such that $H_{-v_1}  = H_{v}$, then as above, 
   $-v_1 = cv$ for some $c > 0$. If we assume $||v_1|| = 1$, then we must have $v_1 = \frac{-v}{||v||} $. 
\end{proof}

\begin{lemma} \label{halfspaceintersect}
Let $V$ be a Euclidean space with positive definite inner product \\$<- , ->$ and let $u, v \in V$. Then 
$H_u^{int} \cap H_v^{int}$ is non-empty if $ u \not= -cv, \ c > 0$. 
\end{lemma}

\begin{proof} Let $u, v \in V$. Clearly $H_u^{int} \cap H_v^{int} = \emptyset$ if $u = -cv$ for some $c > 0$ in $\Bbb{R}$. Lets assume that  $u \not= -cv, c > 0$.  Suppose  that $H_u^{int} \cap H_v^{int} = \emptyset$, then if $x \in H_u^{int}$ we must have $x \not\in H_v^{int}$
and $<x, v> \geq 0$ and $H_u^{int}  \subseteq H_{-v}$. Taking closures, we get $H_u \subseteq H_{-v}$ and by Lemma  \ref{hyper} $H_u = H_{-v}$ and $u = -cv$ for some $c > 0$ giving us a contradiction. 
\end{proof}

\begin{lemma} \label{closure}Let $V$ be a Euclidean space with positive definite inner product \\$<- , ->$ and let $u, v \in V$ for which $ u \not= -cv, \ c > 0$. Then 
 $$\overline{H_u^{int} \cap H_v^{int}} = H_u \cap H_v.$$
\end{lemma}
\begin{proof} Let $u$ and $v$ be as in the statement of  the lemma and let $C = H_u^{int} \cap H_v^{int}$.
Clearly the lemma is true if $u = cv$ for some $c > 0$. Therefore we will assume that $u \not= cv$ for any $c \in \Bbb{R}$ for the rest of the proof. In this case $H(u) \not= H(v)$. 

 Since $\overline{H_u^{int}} = H_u$ and $\overline{H_v^{int}} = H_v$, we have $C
\subseteq H_{u} \cap H_{v}$ and since $H_{u} \cap H_{v}$ is a closed set, we have $\bar{C} \subseteq H_{u} \cap H_{v}$. 
  To verify the opposite inclusion, consider 
 $x \in H_{u} \cap H_v$. If $x \in C$, then $x \in  \bar{C}$. Therefore we need only consider points $x$ where $<x, u> = 0$ or $<x, v> = 0$ or both. If 
 $<x, u> = 0$ and $<x, v> >  0$, we have a sequence $\{x_i\}_{i = 1}^{\infty}$ in $H_{u}^{int}$ with 
 $\lim_{i \to \infty} x_i = x$. Now since $x \in H_{v}^{int}$ which is an open set, there exists  an $m$ for which  $x_i \in H_{v}^{int}$ for all $i \geq m$. Therefore $x_i \in C$ for all $i \geq m$ and $x \in \bar{C}$. A similar argument shows that 
 if $<x, v>  =   0$ and $<x, u>  < 0$, we have $x \in \bar{C}$. 
 
 Now if $<x, v>  =   0$ and $<x, u>  =  0$, then $x \in [H(u) \cap H(v)]$ which is a subspace of $H(v)$ of dimension $n - 2$ since 
 $H(u) \not= H(v)$. 
 Let $H =  [H(u) \cap H(v)]$. It remains to show that $H \subseteq \bar{C}$. 
 We consider  $H$ as  a hyperplane in $H(v)$. It is not difficult to see that if $y \in H(v)$, $<y, u> = < y, \pi(u)>$ where $\pi(u)$ is the orthogonal projection of  $u$ onto the hyperplane $H(v)$. Therefore 
 $$H = \{y \in H(v) | <y, \pi(u)> = 0 \} = 
 H^{H(v)}(\pi(u))$$
  as a hyperplane in $H(v)$.
  Now looking at the corresponding open half space in  $H(v)$, we have $H_{\pi(u)}^{H(v){int}}  = \{x \in H(v)| <x,\pi(u)> < 0\}$. By our reasoning above we know this is a subset of $\bar{C}$. Taking the closure within $H(v)$, we have 
  $H^{H(v)}_{\pi(u)} = \overline{H_{\pi(u)}^{H(v)int} }\subseteq \bar{C}$  and $H = H^{H(v)}(\pi(u)) 
 \subseteq \bar{C}$. This completes the proof.

\end{proof}

\begin{definition} A  pointed convex cone $D$ in a Euclidean space $V$  is called {\bf maximal} if $D \cap (-D) = \{0\}$ and $D \cup (-D) = V$. 
\end{definition}

Note that a maximal pointed convex cone  in a Euclidean space $V$   defines a total ordering on $V$ given by $y \leq x$ if and only if $y - x \in D$.  Also $D = \{v \in V | v \leq 0\}$. 

\begin{definition}
A closed half space $H_u$  is called a {\bf support  for a  cone} $C$ if $C \subseteq H_u$.
\end{definition}

\begin{theorem} \label{unique} Let $D$ be a maximal pointed convex cone in a Euclidean space $V$  with  positive definite inner product $<-, - >$. 
Then $D$ has a unique support  $H_v$,  in $V$ and $H_v = \bar{D}$. Furthermore  $H_v^{int} = \{x \in V | <x, v> < 0 \} \subset D$. 
\end{theorem}
\begin{proof} According to Corollary 1 to Theorem 1 of Fenchel \cite{Fenchel},  a maximal pointed convex cone $D$  in $V$,  has a support $H_{v}$, where 
$D \subseteq H_{v}$. Now since $D \cup (-D) = V$ we have $H_v^{int} = \{x \in V | <x, v> < 0 \} \subseteq D$. Otherwise we would have $x \in (-D)$ with $<x, v> < 0$ and thus $<-x, v> > 0$ for $-x \in D$ giving us a contradiction. Thus we have 
$\overline{H_v^{int}}  =  H_{v} \subseteq \bar{D}$.  On the other hand $\bar{D} \subseteq H_{v}$  since $D \subseteq H_{v}$,
therefore $\bar{D} = H_v$. 

 Suppose that $H_{u}$ is also a support. Then as above, we have $\bar{D} = H_u$ and hence $H_u = H_v$. 
\end{proof}

\begin{theorem} \label{equivalence} Let $V$ be a finite dimensional Euclidean  space of dimension $n$ with positive definite inner product $< - , - >$. The following are equivalent:
\begin{enumerate}
\item D is a maximal pointed convex cone in V,
\item There exists an ordered basis $\{v_1, v_2, \dots , v_n \}$ of $V$ with $$D = \{c_1v_1 + c_2v_2 + \dots + c_nv_n | c_i \in \Bbb{R},  c_m < 0 \ \mbox{where} \ m = \mbox{max}\{i | 
c_i \not= 0 \} \} \cup \{0\},$$
\item There exists an ordered  orthonormal basis $\{v_1, v_2, \dots , v_n \}$ of $V$ with $$D = \{c_1v_1 + c_2v_2 + \dots + c_nv_n | c_i \in \Bbb{R},  c_m < 0 \ \mbox{where} \ m = \mbox{max}\{i | 
c_i \not= 0 \} \}  \cup \{0\}.$$
\end{enumerate}
Moreover the map taking $D$ to $\{v_1, v_2, \dots , v_n \}$ is a bijection from the maximal pointed convex cones in $V$ to the set of  ordered  orthonormal bases of $V$. 
The orthogonal  group $O(V)$ acts simply transitively on the set of maximal pointed convex cones in $V$. 
\end{theorem}

\begin{proof}  The result is obvious if $n = 1$, since the only maximal pointed convex cones of $\Bbb{R}e_1$ are 
$\Bbb{R}_{\geq 0}e_1$ and $\Bbb{R}_{\geq 0}(-e_1)$. We use induction on the dimension of $V$. 

Let $D$ be   a maximal pointed convex cone   in $V$. From Theorem \ref{unique}  we know that  $D$   has a unique  support $H_{v_{n }}$, where 
$D \subseteq H_{v_{n }}$ and $H_{v_n}^{int}  \subseteq D$. From Lemma \ref{hyper} we know that we can choose a $v_n$ which is  unique with the property that 
  $-v_n \in D$,    $||v_n|| = 1$ and  $D \subseteq H_{v_{n }}$.

Now consider the set $\widetilde{D} = D \cap H(v_{n })$. It is not difficult to see that $\widetilde{D}$ is a maximal pointed convex cone in 
$H(v_{n })$.   By induction, we get an orthonormal  basis $\{v_1, v_2, \dots , v_{n - 1 }\}$ for $H(v_{n })$ such that  
$$\widetilde{D} =  \{x \in H(v_{n }) | <x, v_m> < 0 \ \mbox{for} \ m = \mbox{max}\{i | <x, v_m> \not= 0\}, 1\leq i \leq n -1 \} \cup \{0\}.$$
Now   $D = \{x \in V | <x, v_{n }>  < 0\} \cup \widetilde{D}$ and  
$$D =   \{x \in V | <x, v_m> < 0 \ \mbox{for} \ m = \mbox{max}\{i | <x, v_m> \not= 0\} , 1\leq i \leq n \} \cup \{0\}$$
where  $\{v_1, v_2, \dots , v_n\}$ is an orthonormal basis for $V$. 
This proves that (1) implies (3). 

That (3) implies (2) is trivial.  To show that (2) implies (1), let $\{v_1, v_2, \dots , v_n \}$ be an ordered basis of $V$ and let  $$D = \{c_1v_1 + c_2v_2 + \dots + c_nv_n | c_m < 0 \ \mbox{where} \ m = \mbox{max}\{i | 
c_i \not= 0 \} \}  \cup \{0\}.$$ 
It is trivial to check that $D$ is a maximal pointed convex cone in $V$. Therefore (2) implies (1). 

Let $D$ be a maximal pointed convex cone in $V$ with 
$$D = \{c_1v_1 + c_2v_2 + \dots + c_nv_n | c_m < 0 \ \mbox{where} \ m = \mbox{max}\{i | 
c_i \not= 0 \} \}  \cup \{0\} $$
where $\{v_1, v_2, \dots , v_n\}$ is an orthonormal basis of $V$. 
Suppose $\{u_1, u_2, \dots , u_n\}$ is another orthonormal basis of $V$ for which 
$$D = \{c_1u_1 + c_2u_2 + \dots + c_nu_n | c_m < 0 \ \mbox{where} \ m = \mbox{max}\{i | 
c_i \not= 0 \} \}  \cup \{0\}. $$

We have both $H_{u_n}$ and $H_{v_n}$ are supporting half spaces  for $D$. Therefore,
by Theorem \ref{unique} we have $H_{u_n} = H_{v_n}$ and since $||u_n|| = ||v_n|| = 1$, by Lemma \ref{hyper}, $u_n = v_n$.  Now $ \{u_1, u_2, \dots , u_{n-1} \}$ and $ \{v_1, v_2, \dots , v_{n-1} \}$ are ordered orthonormal bases of $H = H(u_n) = H(v_n)$ and $\widetilde{D} = D \cap H $ is a maximal pointed convex cone in $H$ with 
$$\widetilde{D} = \{c_1v_1 + c_2v_2 + \dots + c_{n- 1}v_{n-1} | c_m < 0 \ \mbox{where} \ m = \mbox{max}\{i | 
c_i \not= 0 \} \}  \cup \{0\}$$
$$= \{c_1u_1 + c_2u_2 + \dots + c_{n- 1}u_{n-1} | c_m < 0 \ \mbox{where} \ m = \mbox{max}\{i | 
c_i \not= 0 \} \}  \cup \{0\}.$$
Hence by induction, we have $ \{u_1, u_2, \dots , u_{n-1} \}$ and $ \{v_1, v_2, \dots , v_{n-1} \}$ are equal as 
ordered orthonormal bases of $H$  and hence $ \{u_1, u_2, \dots , u_{n-1}, u_n \}$ and\\
 $ \{v_1, v_2, \dots , v_{n-1}, v_n \}$ are equal as 
ordered orthonormal bases of $V$.  This shows that the map taking maximal pointed convex cones to ordered orthonormal bases given above is one to one. The map is onto since (3) implies (1). 

Since the orthogonal group $O(V)$ acts simply transitively on the set of ordered orthonormal bases of $V$, and 
by linearity the action commutes with the above map from ordered orthonormal bases to maximal pointed convex cones in $V$, 
we have a simple transitive action of $O(V)$ on the set of maximal  pointed convex cones in $V$. 
\end{proof}

\section{The poset $\mathcal{L}_n$.}

For $n \geq 0$ we let $\mathcal{L}_n$ denote  the following  poset of subsets of $S^n$ ordered by inclusion:
$$\mathcal{L}_n = \{T(\Phi^-_n) \cap \Phi_n^+ | T\in O_{n + 1} \} .$$

Note that $\Phi_n^+ = T(\Phi^-_n)\cap \Phi_n^+$ when $T = -I_{\Bbb{R}^{n + 1}}$.   This  is a maximal element of $\mathcal{L}_n$. 
If $T \in O_{n + 1}$ with $T(e_i) = v_i$, $1 \leq i \leq n + 1$, we let 
$$D_T = \{ v = \sum_{i = 1}^{n + 1}c_iv_i \in \Bbb{R}^{n + 1} | <v, v_m> = c_m < 0 \ \mbox{where } \  m = \mbox{max}\{i | c_i \not= 0\}\} \cup \{0\}.$$
It is not difficult to show that $D_T $ is a maximal pointed convex cone in $\Bbb{R}^{n + 1}$ with 
$D_T \cap S^{n} =  T(\Phi_n^-)$. 

\begin{lemma} \label{mpcc} Let $X \in \mathcal{L}_n$, with $X = T_X(\Phi_n^-) \cap \Phi_n^+$, then there is a unique maximal 
pointed convex   cone 
$D_X = D_{T_X}$ in $\Bbb{R}^{n + 1}$ with the property that $X = D_X \cap \Phi_n^+$.
\end{lemma}
\begin{proof} From its definition, we have that $D_{T_X}$ is a maximal pointed convex cone in $\Bbb{R}^{n + 1}$ with  $D_{T_X} \cap \Phi_n^+ =  T_X(\Phi_n^-)\cap \Phi_n^+  = X$. Suppose that $D$ is another maximal pointed convex cone in 
$\Bbb{R}^{n + 1}$ with $X = D \cap \Phi_n^+$. If  $D_e = D_{-I_{\Bbb{R}^{n + 1}}}$, then $\Phi_n^+ = D_e \cap S^n$ and 
$X = D_X \cap D_e \cap S^n = D \cap D_e \cap S^n$. Now $C_1 = D_X \cap D_e$ and $C_2 = D \cap D_e$  are both convex  cones in $\Bbb{R}^{n + 1}$ with $X = C_1 \cap S^n = C_2 \cap S^n$. By Lemma \ref{hatcone}, $\hat{X} = C_1 = C_2$. Let $C'_1 = (D_e \backslash C_1) \cup \{0\} = D_e \cap  (-D_X)$.
 Now  
 $$[C_1 \cup (-C'_1) ]
\cup [C'_1 \cup (-C_1)] = D_e \cup (-D_e) =  \Bbb{R}^{n + 1}$$ and
 $$[C_1 \cup (-C'_1) ]
\cap [C'_1 \cup (-C_1)] = \{0\}$$
since $C_1 \cup (-C'_1)  \subseteq D_X$ and $C'_1 \cup (-C_1) \subseteq -D_X$. Since $D_X \cap [C'_1 \cup (-C_1)] = \{0\}$, we can conclude that $D_X = C_1 \cup (-C'_1)$.

 Since $C'_1 = (D_e \backslash C_2) \cup \{0\} = D_e \cap  (-D)$, we can similarly conclude that $D = C_1 \cup (-C'_1)$ and hence $D = D_X$. 

\end{proof}

\begin{definition} \label{mpccX} We define    $D_e = D_{-I_{\Bbb{R}^{n + 1}}}$, that is 
$$D_e =  \{ x = \sum_{i = 1}^{n + 1}c_ie_i \in \Bbb{R}^{n + 1} | <x, e_m> = c_m > 0 \ \mbox{where } \  m = \mbox{max}\{i | c_i \not= 0\}\} \cup \{0\}$$
and $D_e \cap S^n = \Phi_n^+$. 
For  $X \in \mathcal{L}_n$ we let $D_X$ denote the unique maximal pointed convex cone in $\Bbb{R}^{n + 1}$ for which $X = D_X \cap \Phi_n^+ = D_X \cap D_e \cap S^n$. 
\end{definition}

 Therefore we can characterize $\mathcal{L}_n$ as
$$\mathcal{L}_n = \{D\cap \Phi_n^+ | D \ \mbox{is a maximal pointed convex cone in } \ \Bbb{R}^{n + 1}\}.$$
We will alternate between these characterizations of $\mathcal{L}_n$ as is convenient. 

\begin{example}
 If $n = 0$, then $\Bbb{R}^{n + 1} = \Bbb{R}^1 = \Bbb{R}e_1$. We have $S^0 = \{-e_1, e_1\}, \Phi_0^+ = \{e_1\}$ and $O_1 
= \{\pm 1_{e_1}\}$.  Therefore $\mathcal{L}_0 = \{\emptyset, \{e_1\}\}$. 
\end{example}

\begin{lemma}\label{duality} Let 
 $X \in \mathcal{L}_n$, then $\Phi^+_n \backslash X$ is also an element of $\mathcal{L}_n$.  The map $\phi: \mathcal{L}_n \to \mathcal{L}_n^{\delta}$ given by 
$\phi(X) = \Phi^+_n \backslash X$ is an order  isomorphism of posets. 
\end{lemma}
\begin{proof} Let 
 $X \in \mathcal{L}_n$. 
 There exists a maximal pointed convex cone $D$  in $\Bbb{R}^{n + 1}$ for which $X = D \cap \Phi_n^+$. 
 We have that $(-D)$ is also a maximal pointed convex cone in $\Bbb{R}^{n + 1}$ and $(-D)\cap  \Phi_n^+ = \Phi_n^+ \backslash X$. 
 Hence $\Phi_n^+ \backslash X \in \mathcal{L}_n$. Since $\phi^2(X) = X$, it is easy to see that $\phi$ is a bijective map. 
 Since $ \Phi_n^+ \backslash X \subseteq  \Phi_n^+ \backslash Y$ if and only if $X \supseteq Y$, we have an isomorphism of posets 
$\phi: \mathcal{L}_n \to \mathcal{L}_n^{\delta}$. 
\end{proof}

\begin{theorem} \label{projection}
Let $X$ be an element of $\mathcal{L}_n$, then $X \cap \Bbb{R}^{n }$ is an 
element of $\mathcal{L}_{n - 1}$, where $\Bbb{R}^n $ is imbedded  in   $\Bbb{R}^{n + 1} $ as in Section \ref{intro}.
\end{theorem}

\begin{proof} Given $X \in \mathcal{L}_n$, 
we have a maximal pointed convex cone  $D \subset \Bbb{R}^{n + 1}$ such that $X = D \cap  \Phi_n^+$.  Let $\widetilde{D} = D  \cap  \Bbb{R}^{n }$. It is not difficult to see that $\widetilde{D}$ is a maximal pointed convex cone  of $\Bbb{R}^{n }$ and therefore $X \cap  \Bbb{R}^{n }  = D  \cap  \Bbb{R}^{n } \cap \Phi_{n - 1}^+  = \widetilde{D}  \cap \Phi_{n - 1}^+  \in \mathcal{L}_{n - 1}$. 
\end{proof}

We will prove a series of technical lemmas which will be helpful when constructing the join of two elements in $\mathcal{L}_n$.

\begin{lemma} \label{boundary} Let $X \in \mathcal{L}_n$, with $X = T_X(\Phi^-_n)\cap \Phi^+_n$ for some $T_X \in O_{n + 1}$.
Let $u_X = T_X(e_{n + 1})$ and $e = -e_{n + 1}$.
 If $u_X \not=   e_{n + 1} = -e$, 
then 
\begin{enumerate}
\item $X = (X\cap H(u_X)) \cup \{x \in \Phi_n^+ | <x, u_X> < 0 \},$
\item $\bar{X} =  [H_{u_X} \cap H_{e}] \cap S^n$,
\item $\bar{X} \cap S^{n -1} = H_{u_X}\cap S^{n - 1} \supseteq H(u_X)\cap S^{n - 1}$.
\end{enumerate}
\end{lemma}
\begin{proof}
(1) follows from the fact that $X = T_X(\Phi^-_n) \cap \Phi^+_n$ and that $$
T_X(\Phi_n^-) = \{x \in S^n| x = \sum_{i = 1}^{n + 1}c_iT_X(e_i) \ \mbox{such that } \ c_m < 0 , \ \mbox{where} \ 
m = \mbox{max}\{i| c_i \not= 0\}\}.$$

Assertion (3) follows from (2) since, given (2), we have 
$$\bar{X}\cap S^{n -1} = [H_{u_X}\cap S^{n - 1}] \cap
[H_{e} \cap S^{n - 1}] \cap [S^n  \cap S^{n - 1}] = [H_{u_X}\cap S^{n - 1}]  \cap S^{n - 1}.$$ 

It remains to prove (2). 

 Let $C$ be the convex cone $[H_{u_X}^{int} \cap H_{e}^{int}] \cup \{0\}$, the intersection of the two open half spaces with the zero vector attached.  $C$ is non-empty by Lemma \ref{halfspaceintersect}.
Now  $C \cap S^{n} \subset X \subseteq [H_{u_X} \cap H_{e}]\cap S^n $ and hence $$
\overline{C \cap S^{n} }  \subseteq \overline{X}  \subseteq [H_{u_X} \cap H_{e}]\cap S^n .$$ 
By Lemma \ref{cone} 
$$\overline{C} \cap S^{n}   \subseteq \overline{X}  \subseteq [H_{u_X} \cap H_{e}]\cap S^n. $$
 By Lemma \ref{closure}, $\bar{C} = H_{u_X} \cap H_{e}$.
Therefore  $\overline{X}  =  [H_{u_X} \cap H_{e}]\cap S^n$
and the result follows.  
 \end{proof}

 \begin{lemma} \label{equality} 
  Let $X$, $Y \in \mathcal{L}_n$ with $X = T_X(\Phi_n^-) \cap \Phi_n^+$ and 
  $Y = T_Y(\Phi_n^-) \cap \Phi_n^+$. 
  Let $u_X = T_X(e_{n + 1})$  and let $u_Y = T_Y(e_{n + 1})$. 
  If $H(u_Y) \cap \Bbb{R}^{n } \not= \Bbb{R}^{n }$ and 
  $H(u_X) \cap \Bbb{R}^{n } \subseteq H(u_Y) \cap \Bbb{R}^{n }$, then 
  $$H(u_X) \cap \Bbb{R}^{n } = H(u_Y) \cap \Bbb{R}^{n }.$$
  \end{lemma}
  \begin{proof}
  We have $ H(u_Y)\cap \Bbb{R}^{n } $ is a subspace of $\Bbb{R}^{n}$ of dimension $n$ or of dimension $n - 1$ . Since we are assuming that  $H(u_Y) \cap \Bbb{R}^{n } \not= \Bbb{R}^{n }$, it must be a subspace of dimension $n - 1$. Since $H(u_X) \cap \Bbb{R}^n \subseteq H(u_Y) \cap \Bbb{R}^n$ it too must be a subspace of dimension $n - 1$ and hence must be all of $H(u_Y) \cap \Bbb{R}^n$. 
  \end{proof}

\begin{lemma}\label{Z}
Let $Z \in \mathcal{L}_n$, with $Z = T_Z(\Phi_n^-)\cap \Phi_n^+$ for some $T_Z \in O_{n + 1}$. 
Let $u_Z = T_Z(e_{n + 1})$.
If $H(u_Z) = \Bbb{R}^n$,   then 
$T_Z = T'_Z \bot \pm 1_{e_{n + 1}}$ for some $T'_Z \in O_n$.

\end{lemma}
\begin{proof}
 Since $T_Z$ is an isometry, with $T_Z(\Bbb{R}^n) = H(u_Z) =  \Bbb{R}^n$,  we have 
 $T_Z|_{\Bbb{R}^n} = T'_Z \in O_n$. Now $||u_Z|| = 1$ and $<u_Z, x> = 0$ for all $x \in  \Bbb{R}^n$.  Hence $u_Z = \pm e_{n + 1}$ and 
$T_Z = T'_Z\bot \pm 1_{e_{n + 1}}$. 
\end{proof}

\section{The least upper  bound of two elements of $\mathcal{L}_n$. }

We will use induction to find $X \vee Y$  for two elements $X$ and $Y$ of $\mathcal{L}_n$. Obviously we can find 
$X\vee Y$ in $\mathcal{L}_0$ for any pair of elements $\{X, Y\}$ in $\mathcal{L}_0$.

\begin{theorem}  \label{conditions}
Let $X, Y \in \mathcal{L}_n$ with $X = T_X(\Phi_n^-) \cap \Phi_n^+$ and $Y = T_Y(\Phi_n^-) \cap \Phi_n^+$, where $T_X, T_Y \in O_{n + 1}$. 
Let $u_X = T_X(e_{n + 1})$ and $u_Y = T_Y(e_{n + 1})$. 
Assume that
$H(u_X) \not= \Bbb{R}^{n}$ and $H(u_Y) \not= \Bbb{R}^{n}$. 

Then if $H(u_X) \cap \Bbb{R}^{n} \not= H(u_Y) \cap \Bbb{R}^{n}$
and  $Z \in \mathcal{L}_n$ has the property that $X\subseteq Z$ and $Y\subseteq Z$, where 
$Z = T_Z(\Phi_n^-)\cap \Phi_n^+$ for some $T_Z \in O_{n + 1}$,  we must have 
$$T_Z(\Bbb{R}^{n}) = \Bbb{R}^{n} \ \ \mbox{and} \ \ T_Z = T_Z' \bot (-1_{e_{n + 1}}).$$

\end{theorem}

\begin{proof} Let $Z \in \mathcal{L}_n$ with  the property that $X\subseteq Z$ and $Y\subseteq Z$, where 
$Z = T_Z(\Phi_n^-)\cap \Phi_n^+$ for some $T_Z \in O_{n + 1}$. Let $u_Z = T_Z(e_{n + 1})$.
By Theorem \ref{boundary}  we have   $\bar{Y}\cap  \Bbb{R}^n = H_{u_Y} \cap S^{n - 1}$ and $\bar{Z}\cap  \Bbb{R}^n = H_{u_Z} \cap S^{n - 1}$ and since $\bar{Y} \subseteq \bar{Z}$, we must have $H_{u_Y} \cap S^{n - 1} \subseteq H_{u_Z} \cap S^{n - 1}$. If $H_{u_Z} \cap S^{n - 1} \not= S^{n - 1}$ then we must have that $H({u_Y}) \cap S^{n - 1} \subseteq H({u_Z}) \cap S^{n - 1}$, otherwise there exists $y \in H({u_Y}) \cap S^{n - 1}$ with $< y, u_Z> \not= 0$ and either $y$ or its antpodal point $-y$ is not in $ H_{u_Z} \cap S^{n - 1}$ giving a contradiction. 
Therefore by Lemma \ref{equality} we have $H({u_Y}) \cap S^{n - 1} = H({u_Z}) \cap S^{n - 1}$. By a similar argument 
we have $H({u_X}) \cap S^{n - 1} = H({u_Z}) \cap S^{n - 1}$. However 
by assumption, $H(u_X) \cap \Bbb{R}^{n} \not= H(u_Y) \cap \Bbb{R}^{n}$ and we have a contradiction. Therefore, we must have 
$$S^{n - 1} = \bar{Z}\cap  \Bbb{R}^n = H_{u_Z} \cap S^{n - 1} = H(u_Z) \cap S^{n - 1} \cup \{x \in S^{n - 1} | <x,  u_Z> < 0\}.$$
 If $x \in S^{n - 1}$ with $ <x,  u_Z> < 0$, then $-x \in S^{n - 1}$ and $<-x ,  u_Z> > 0$. Therefore $\{x \in S^{n - 1} | <x,  u_Z> < 0\} = \emptyset $ and $S^{n - 1} =  H(u_Z) \cap S^{n - 1}$. Consequently $H(u_Z)  = \Bbb{R}^n$ 
and $T_Z = T_Z' \bot \pm1_{e_{n + 1}}$  by Lemma \ref{Z}.  Now $Y \cap \{x \in S^n | <x, e_{n + 1}> > 0\} \not= \emptyset$ since $H(u_Y) \not= \Bbb{R}^{n}$. Thus $Z\cap \{x \in S^n | <x, e_{n + 1}> > 0\} \not= \emptyset$ since $Y \subseteq Z$. Therefore 
$T_Z = T_Z' \bot (-1_{e_{n + 1}})$. 
 \end{proof}

\begin{theorem}  \label{OTHER}
Let $X, Y \in \mathcal{L}_n$ with $X = T_X(\Phi_n^-) \cap \Phi_n^+$ and $Y = T_Y(\Phi_n^-) \cap \Phi_n^+$, where $T_X, T_Y \in O_{n + 1}$. Let $u_X = T_X(e_{n + 1})$ and let $u_Y = T_Y(e_{n + 1})$. 
Assume that
$H(u_X) \not= \Bbb{R}^{n}$ and $H(u_Y) \not= \Bbb{R}^{n}$. 

If $H(u_X) \cap \Bbb{R}^{n} \not= H(u_Y) \cap \Bbb{R}^{n}$
 there exists $Z \in \mathcal{L}_{n}$ such that $Z = X \vee Y$. 
\end{theorem}
\begin{proof} Let $X_1 = X \cap \Bbb{R}^n$ and $Y_1 = Y  \cap \Bbb{R}^n$. Then
 by Theorem \ref{projection} we have 
 $X_1   \in \mathcal{L}_{n - 1} $ and $Y_1  \in \mathcal{L}_{n - 1}$.  We proceed by induction. Let $Z_1 \in \mathcal{L}_{n -1}$ where  $Z_1$ is the least upper bound of $\{X_1, Y_1\} $ in $
\mathcal{L}_{n -1}$.  Then  $Z_1 = D_{Z_1} \cap \Phi_{n - 1}^+$ for  a unique maximal pointed convex cone of $\Bbb{R}^n$. 

Let $D_Z = D_{Z_1} \cup \{x \in \Bbb{R}^{n + 1} | <x, e_{n + 1}> > 0\}$. It is not difficult to see that $D_Z$ is a maximal pointed convex cone in $\Bbb{R}^{n + 1}$ with $D_Z \cap \Bbb{R}^n = D_{Z_1}$.
We let $Z = D_Z \cap \Phi_n^+ \in \mathcal{L}_n$. Then 
$Z\cap S^{n - 1} = Z_1$ and therefore 
 $$X = X_1 \cup [X \cap \{x \in S^n | <x, e_{n + 1}> > 0\}] \subseteq Z$$
 and $$Y = Y_1 \cup [Y \cap \{x \in S^n | <x, e_{n + 1}> > 0\}] \subseteq Z.$$  

 Now, if $W \in \mathcal{L}_n$ with the property that both $X$ and $Y$ are subsets of $W$, then 
Theorem \ref{conditions} tells us that $\{x \in S^n | <x, e_{n + 1}> > 0\} \subseteq W$. Since $W\cap S^{n - 1}$ contains both $X_1$ and $Y_1$ and,
according to Theorem \ref{projection},
$W\cap S^{n - 1} \in \mathcal{L}_{n - 1}$,   we must have $Z_1 = Z\cap S^{n - 1} \subset W\cap S^{n - 1} \subseteq  W$. Therefore $Z \subseteq W$ and $Z$ is the least upper bound of $\{X, Y\}$ in $\mathcal{L}_n$. 
\end{proof}

\begin{theorem}\label{BAR}  Let $X, Y \in \mathcal{L}_n$ with $X = T_X(\Phi_n^-)\cap \Phi_n^+$ and $Y = T_Y(\Phi_n^-)\cap \Phi_n^+$. If $\bar{X} \subseteq \bar{Y}$, then there exists $Z \in \mathcal{L}_{n}$ such that $Z$ is the least upper bound of the set $\{X, Y\}$. 
\end{theorem}

\begin{proof} Let $X$ and $Y$ be as in the statement of the theorem with $\bar{X} \subseteq \bar{Y}$. Then 
$X = D_X \cap \Phi_n^+$ and $Y = D_X \cap \Phi_n^+$ where $D_X$ and $D_Y$ are maximal pointed convex cone in $\Bbb{R}^{n + 1}$. 
Suppose  $Z =  T_Z(\Phi_n^-)\cap \Phi_n^+  =  D_Z\cap \Phi_n^+ \in \mathcal{L}_n$ is  an upper bound for $\{X, Y\}$, where $T_Z \in O_{n + 1}$ and $D_Z$ is a maximal pointed convex cone in $\Bbb{R}^{n + 1}$. 
 Now  $Z = X \vee Y$ if and only if $X$ and $Y$ are contained in $Z$ and if $D_1$ is any  maximal pointed convex cone in 
$\Bbb{R}^{n + 1}$ for which $X$ and $Y$ are contained in $D_1 \cap \Phi^+_n$, then $Z = D_Z \cap  \Phi^+_n \subseteq D_1 \cap  \Phi^+_n$. 
We will use this characterization of $X\vee Y$ and  induction on $n$ to show that $Z = X\vee Y$ exists in $\mathcal{L}_n$. 
We have already noted that  we can find the join of any two elements in $\mathcal{L}_0$.

Let $u_Y = T_Y(e_{n + 1})$ and recall the definition of $D_e$ from Definition \ref{mpccX}. We first work on constructing that part of the boundary of $Z$ which meets $H(u_Y)$.
 Let 
 $$\dddot{\Phi} = \Phi_n^+ \cap H(u_Y) = D_e \cap H(u_y) \cap S^{n} = \dddot{D_e}\cap S^{n },$$
 where $\dddot{D_e} = D_e \cap H(u_y)$. We let 
 $$\dddot{X} = X \cap H(u_Y) = D_X \cap H(u_Y) \cap \Phi_n^+ = \dddot{D_X} \cap \Phi_n^+ =  \dddot{D_X} \cap \dddot{\Phi}$$
 where $\dddot{D_X} = D_X \cap H(u_Y)$ and  let 
  $$\dddot{Y} = Y \cap H(u_Y)
=D_Y \cap H(u_Y) \cap \Phi_n^+ =  \dddot{D_Y} \cap \Phi_n^+ =  \dddot{D_Y} \cap \dddot{\Phi}$$
 where $\dddot{D_Y} = D_Y \cap H(u_Y)$.
  Now $\dddot{D}_X$,  $\dddot{D}_Y $ and $\dddot{D_e}$ are  maximal pointed convex cones of $H(u_Y)$.
 Therefore  $T_Y^{-1}(\dddot{D_e})$     is a  maximal pointed convex cone in $\Bbb{R}^n$.  Now by Theorem \ref{equivalence}
 There exists an ordered  orthonormal basis $\{v_1, v_2, \dots , v_n \}$ of $\Bbb{R}^n$ with 
 $$T_Y^{-1}(\dddot{D_e}) = \{c_1v_1 + c_2v_2 + \dots + c_nv_n | c_m < 0 \ \mbox{where} \ m = \mbox{max}\{i | 
c_i \not= 0 \} \} \cup \{0\}.$$ Let $T_1$ be the orthogonal transformation in $O_n$ that sends $e_i$ to $-v_i$, $1 \leq i \leq n$. Then $$T_1^{-1}(T_Y^{-1}(\dddot{\Phi})) = T_1^{-1}(T_Y^{-1}(\dddot{D_e}\cap S^{n })) = T_1^{-1}(T_Y^{-1}(\dddot{D_e}))\cap T_1^{-1}(T_Y^{-1}(S^{n })) = \Phi^+_{n - 1}.$$
 Letting $T :H(u_Y) \to \Bbb{R}^n$ denote the isometry $T = T_1^{-1}T_Y^{-1} $, we get $T(\dddot{\Phi}) = \Phi^+_{n - 1}$ and 
$$
\widetilde{X} = T(\dddot{X})
= T(\dddot{D_X}) \cap T(\dddot{\Phi}) = 
T(\dddot{D_X}) \cap \Phi_{n - 1}^+ \in \mathcal{L}_{n -1}.
$$
since $T(\dddot{D_X})$ is a maximal pointed convex cone in $\Bbb{R}^n$. 
 Similarly $\widetilde{Y} =  T(\dddot{Y}) \in \mathcal{L}_{n -1}$. 
Using our inductive assumption, we can find  $\widetilde{Z} \in \mathcal{L}_{n - 1}$ with $\widetilde{Z} = \widetilde{X}\vee \widetilde{Y}$.  Now $\widetilde{Z} = D_{\widetilde{Z}} \cap \Phi_{n - 1}^+$ for some maximal pointed convex cone $D_{\widetilde{Z}}$ in $\Bbb{R}^n$ and if $D$ is another maximal pointed convex cone in $\Bbb{R}^n$ for which $ D \cap \Phi_{n - 1}^+$ is an upper bound for $\{\widetilde{X}, \widetilde{Y}\}$, then $ D_{\widetilde{Z}} \cap \Phi_{n - 1}^+ \subseteq D \cap \Phi_{n - 1}^+$.  The isometry $T^{-1} : \Bbb{R}^n \to H(u_Y)$ gives a bijection between maximal pointed convex cones in $\Bbb{R}^n$ and maximal pointed convex cones in 
$H(u_Y)$.  Therefore 
$$\dddot{Z} = T^{-1}(\widetilde{Z}) = T^{-1}(D_{\widetilde{Z}}) \cap T^{-1}(\Phi_{n - 1}^+) = T^{-1}(D_{\widetilde{Z}}) \cap \dddot{\Phi}$$
 contains $\dddot{X}$ and $\dddot{Y}$ and if $D_1$ is a maximal pointed convex cone in $H(u_Y)$ for which $D_1 \cap  \dddot{\Phi}$ contains $\dddot{X}$ and $\dddot{Y}$, then $\dddot{Z} \subseteq D_1 \cap  \dddot{\Phi}$. We let 
$\dddot{D_Z} = T^{-1}(D_{\widetilde{Z}})$,  a  maximal pointed convex cone in $H(u_Y)$ with $\dddot{D_Z}  \cap \dddot{\Phi} = \dddot{Z}$ . 

Consider $H_{u_Y}^{int} = \{x \in \Bbb{R}^{n + 1} | <x, u_Y> < 0 \}$. We have $Y = \dddot{Y} \cup (H_{u_Y}^{int} \cap \Phi_{n}^+$). Now 
$$X = \dddot{X} \cup \{x \in X | <x, u_Y> < 0\} = \dddot{X} \cup [X \cap (H_{u_Y}^{int} \cap \Phi_{n}^+)]$$
 since $\bar{X} \subseteq \bar{Y} \subseteq H_{u_Y}$. 
  Let $D_Z = \dddot{D_{Z}} \cup H_{u_Y}^{int}$. It is clear  that $D_Z$ is a maximal pointed convex cone in $\Bbb{R}^{n + 1}$ with 
  $D_Z \cap H(u_Y) = \dddot{D_Z}$. I claim that  $Z = D_Z \cap \Phi^+_n$ is the  least upper bound of $\{X, Y\}$ in $\mathcal{L}_n$. Certainly 
$$X\subseteq \dddot{X} \cup [  H^{int}_{u_Y} \cap \Phi^+_n] 
\subseteq \dddot{Z} \cup [  H^{int}_{u_Y} \cap \Phi^+_n]  \subseteq Z$$
 and similarly $Y \subseteq Z$. 

Now suppose $D'$ is another maximal pointed convex cone of $\Bbb{R}^{n + 1}$ with the property that $D' \cap \Phi_n^+$ contains both $X$ and $Y$. We must show that 
$Z \subseteq D' \cap \Phi_n^+$. We have $\dddot{X} \subseteq X \subseteq D' \cap \Phi_n^+$ and $\dddot{Y} \subseteq Y \subseteq D' \cap \Phi_n^+$. If we let $\dddot{D'} = D' \cap H(u_Y)$, then $\dddot{D'} $ is a maximal pointed cone of $H(u_Y)$ and $\dddot{D'} \cap \dddot{\Phi}$ contains both $\dddot{X}$ and $\dddot{Y}$. Therefore as shown above, $\dddot{Z} = \dddot{D_{Z}}  \cap \dddot{\Phi}   \subseteq \dddot{D'} \cap \dddot{\Phi}  \subseteq D' \cap \Phi_n^+$. Now since 
 $(H_{u_Y}^{int} \cap \Phi_{n}^+)  \subseteq Y \subseteq   D' \cap \Phi_n^+$, we have that $Z =  (\dddot{D_{Z}}  \cap\dddot{\Phi})
 \cup (H_{u_Y}^{int} \cap \Phi_{n}^+)  \subseteq D' \cap \Phi_n^+$. This completes that proof that $Z = X\vee Y$ in $\mathcal{L}_n$. 
\end{proof}

\begin{theorem}Let $X, Y \in \mathcal{L}_n$ with $X = T_X(\Phi_n^-)\cap \Phi_n^+$ and $Y = T_Y(\Phi_n^-)\cap \Phi_n^+$. Then 
there exists $Z \in \mathcal{L}_n$ with the property that $Z$ is a least upper bound for $\{X, Y\}$. 
\end{theorem}
\begin{proof} Let $u_X = T_X(e_{n + 1})$ and $u_Y = T_Y(e_{n + 1})$. If the conditions of theorem \ref{OTHER} or theorem 
 \ref{BAR} are met then there is a least upper bound for $\{X, Y\}$ in $\mathcal{L}_n$. 

We have two remaining cases to consider. First we will consider the case where neither $T_X$ nor $T_Y$ fix $\Bbb{R}^n$ and $H(u_X) \cap \Bbb{R}^{n} = H(u_Y) \cap \Bbb{R}^{n}$. Secondly we will consider the case where either $T_X$ or $T_Y$ fixes $\Bbb{R}^n$ but we have neither $\bar{X} \subseteq \bar{Y}$ nor 
$\bar{Y} \subseteq \bar{X}$. 
\vskip .1in
\noindent
{\bf Case 1.} Let us assume that $H(u_X) \cap \Bbb{R}^{n} = H(u_Y) \cap \Bbb{R}^{n}$, $H(u_X) \cap \Bbb{R}^{n} \not= \Bbb{R}^{n}$  and $H(u_Y) \cap \Bbb{R}^{n} \not= \Bbb{R}^{n}$. 
If $\bar{X} \subseteq \bar{Y}$ or $\bar{Y} \subseteq \bar{X}$, then the case has been covered in Theorem \ref{BAR}. Therefore we can also assume that $\bar{X} \not\subseteq \bar{Y}$ and $\bar{Y} \not\subseteq \bar{X}$. 

We first show that $(\bar{X} \cap \Bbb{R}^n) \cup (\bar{Y} \cap \Bbb{R}^n)  = S^{n - 1}$. 
By Lemma  \ref{boundary}, $\bar{X} = H_{u_X} \cap H_{-e_{n + 1}}\cap S^n$ and $\bar{Y} = H_{u_Y} \cap H_{-e_{n + 1}}\cap S^n$. Since $\bar{X} \not\subseteq \bar{Y}$ and $\bar{Y} \not\subseteq \bar{X}$, we must have 
an $x \in \bar{X}$ with $<x, u_Y>  > 0$ and a $y \in \bar{Y}$ with $<y, u_X>  > 0$. Now  $y \in H_{-u_X}^{int}$ which is an open set,
and $H_{u_Y}^{int} \cap H_{-e_{n + 1}}^{int} \not= \emptyset$. Therefore 
we can assume that $<y, u_Y> < 0$ since 
 $ \bar{Y} = H_{u_Y} \cap H_{-e_{n + 1}}\cap S^n  = \overline{H_{u_Y}^{int} \cap H_{-e_{n + 1}}^{int} \cap S^n}$. 
Thus we have 
$$<-y, u_X>  < 0,  \ \ <-y, u_Y>  > 0\ \mbox{ and} \ <-y, e_{n + 1}>  \leq 0,$$
 and  both $x$ and $-y$ are in the convex cone $ H_{u_X} \cap H_{-u_Y}$. Then we have $\lambda x + (1 - \lambda)(-y) \in 
H_{u_X} \cap H_{-u_Y}$ for all values of $\lambda$ between $0$ and $1$. Since 
$<x, e_{n + 1}> \geq 0$ and $<-y, e_{n + 1}> \leq  0$, we must have $<\lambda x + (1 - \lambda)(-y), e_{n + 1}> = 0$
for some $0 \leq \lambda \leq 1$. For such a $\lambda$, let 
$x_1 = \lambda x + (1 - \lambda)(-y) \in H_{u_X} \cap \Bbb{R}^n$. Since  
$$<x_1, u_Y> = \lambda<x, u_Y> + (1 - \lambda)<-y, u_Y>  > 0,$$
 we have $x_1 \not\in H_{u_Y} \cap \Bbb{R}^n$.  
By assumption, we have $H = H(u_X) \cap \Bbb{R}^{n} = H(u_Y) \cap \Bbb{R}^{n}$ and since $H(u_X) \not= \Bbb{R}^{n}$ and 
$H(u_Y) \not= \Bbb{R}^{n}$, H must be a hyperplane in $\Bbb{R}^{n}$. In fact it is easy to see that $H = H^{\Bbb{R}^n}(\pi(u_X)) = H^{\Bbb{R}^n}(\pi(u_Y))$, where $\pi$ denotes orthogonal projection onto $\Bbb{R}^n$. 
Also by lemma \ref{hyper}, we have $\pi(u_X) = c\pi(u_Y)$ for some $c \in \Bbb{R}$.  Since $x_1 \in 
H^{\Bbb{R}^n}_{\pi(u_X)} = H_{u_X} \cap \Bbb{R}^n$ and $x_1 \not\in H^{\Bbb{R}^n}_{\pi(u_Y)} =  H_{u_Y} \cap \Bbb{R}^n$, we  must have $H^{\Bbb{R}^n}_{\pi(u_X)} \not= H^{\Bbb{R}^n}_{\pi(u_Y)}$ and $\pi(u_X) = c\pi(u_Y)$ for some $c < 0$. Therefore
$H^{\Bbb{R}^n}_{\pi(u_X)} \cup H^{\Bbb{R}^n}_{\pi(u_Y)} = \Bbb{R}^n$. 
By Lemma \ref{boundary}, we have $\bar{X} \cap \Bbb{R}^n = H_{u_X} \cap S^{n-1}$ and $\bar{Y} \cap \Bbb{R}^n = H_{u_Y} \cap S^{n-1}$, therefore $(\bar{X} \cap \Bbb{R}^n) \cup (\bar{Y} \cap \Bbb{R}^n)  = S^{n - 1}$.

Let  $Z \in \mathcal{L}_n$ with $Z = T_Z(\Phi_n^-)\cap \Phi^+_n$   which contains both $X$ and $Y$. Let 
$u_Z = T_Z(e_{n + 1})$.  Then $\bar{Z}$ contains both $\bar{X}$ and $\bar{Y}$. 
Since $\bar{Z} = H_{u_Z} \cap H_{-e_{n + 1}}$ and $S^{n - 1} = (\bar{X} \cap \Bbb{R}^n) \cup (\bar{Y} \cap \Bbb{R}^n) \subseteq \bar{Z}$, we  have $H({u_Z)} = \Bbb{R}^n$. Therefore by Lemma \ref{Z}, we have 
$T_Z = T'_Z \bot \pm 1_{e_{n + 1}}$ for some $T'_Z \in O_n$.
 Since we have assumed that $H(u_Y) \cap \Bbb{R}^{n} \not= \Bbb{R}^{n}$, we have $y \in Y \subseteq Z$ with $<y, e_{n + 1}> > 0$. Therefore we can assume that 
$T_Z = T'_Z \bot ( -1_{e_{n + 1}})$. 

Now let $Z_1 \in \mathcal{L}_{n -1}$ be a least upper bound for $X_1 = X\cap \Bbb{R}^n$ and $Y_1 = Y\cap \Bbb{R}^n$. We have $Z_1 = D_{Z_1}\cap \Phi_{n-1}^+$, where $D_{Z_1}$ is a maximal pointed convex cone of 
$\Bbb{R}^n$. Let $D = D_{Z_1} \cup \{x \in S^n | <x, e_{n +1}> > 0\}$. It is not difficult to see that $D$ is a maximal pointed convex cone in $\Bbb{R}^{n + 1}$ and that $D\cap \Bbb{R}^n = D_{Z_1}$. Therefore 
$Z_2 = D \cap \Phi_n^+$ is an element of $\mathcal{L}_n$ with $Z_2 \cap \Bbb{R}^n = Z_1$. 
Since $Z_1$ contains both $X_1$ and $Y_1$ and $Z_2$ contains $Z_1$ and $ \{x \in S^n | <x, e_{n +1}> > 0\}$, we must have that $Z_2$ contains both $X$ and $Y$. Hence $Z_2$ is an upper bound for $\{X, Y\}$ in $\mathcal{L}_n$. I claim that it is the least upper bound of $\{X, Y\}$ in $\mathcal{L}_n$.

If $Z \in \mathcal{L}_n$ with $Z = T_Z(\Phi_n^-)\cap \Phi^+_n$ such that  $Z$ contains both $X$ and $Y$,  we have, from above,  that $T_Z = T'_Z \bot ( -1_{e_{n + 1}})$ for some $T'_Z \in O_n$. Now $Z \cap S^{n -1} = T'_Z(\Phi_{n - 1}^-) \cap \Phi_{n-1}^+$ is an element of $\mathcal{L}_{n - 1}$ which contains both $X_1$ and $Y_1$. Therefore 
$Z_1 \subseteq  Z \cap S^{n -1}$. Since $T_Z = T'_Z \bot ( -1_{e_{n + 1}})$, we have $\{x \in S^n | <x, e_{n +1}> > 0\} \subseteq Z \cap S^{n -1}$. Therefore, we have $Z_2 = Z_1 \cup \{x \in S^n | <x, e_{n +1}> > 0\} \subseteq Z$ and 
$Z_2 = X \vee Y$. 

\vskip .1in
\noindent
{\bf Case 2} We  assume that $T_X(\Bbb{R}^n) = \Bbb{R}^n$,  $\Bar{Y} \not \subseteq \bar{X}$ and $\Bar{X} \not \subseteq \bar{Y}$. 
By Lemma \ref{Z}, we have $T_X = T'_X \bot \pm 1_{e_{n + 1}}$ for some $T'_X \in O_n$. 
If $T_X = T'_X \bot ( -1_{e_{n + 1}})$, we have $\bar{X} = H_{-e_{n + 1}} \cap S^n$ and $\bar{Y} \subseteq \bar{X}$.
Therefore we must have $T_X = T'_X \bot ( 1_{e_{n + 1}})$ and $X \subset S^{n - 1}$.  If $T_Y(\Bbb{R}^n) = \Bbb{R}^n$, then 
by Lemma \ref{Z}, we have $T_Y = T'_Y \bot \pm 1_{e_{n + 1}}$ for some $T'_Y \in O_n$ in which case, either 
$\bar{X} \subseteq \bar{Y}$ or  $T_Y = T'_Y \bot  (1_{e_{n + 1}})$. In the latter case $X \subset S^{n - 1}$ and 
$Y \subset S^{n - 1}$. Therefore both $X$ and $Y$ can be viewed as elements of $\mathcal{L}_{n-1}$ and 
by induction there exists $Z \in \mathcal{L}_{n-1}$ such that $Z$ is a least upper bound for $\{X, Y\}$ in $\mathcal{L}_{n - 1}$.  We have  $Z = T'_Z(\Phi_{n - 1}^-) \cap \Phi_{n - 1}^+$ for some $T'_Z \in O_n$. Letting $T_Z = 
T'_Z \bot 1_{e_{n + 1}} \in O_{n + 1}$, we see that $Z = T_Z(\Phi_n^-) \cap \Phi_n^+$ is  a least upper bound for $X$ and $Y$ in $\mathcal{L}_n$. 

We are left with the case where $T_X = T'_X \bot ( 1_{e_{n + 1}})$,  $X \subset S^{n - 1}$,  $T_Y(\Bbb{R}^n)\not= \Bbb{R}^n$,
$\Bar{Y} \not \subseteq \bar{X}$ and $\Bar{X} \not \subseteq \bar{Y}$. 
  Since  $\bar{X} = \bar{X} \cap \Bbb{R}^n \not\subseteq \bar{Y}  \cap \Bbb{R}^n = H_{u_Y} \cap \Bbb{R}^n$,  we have $x \in \bar{X}$ with $<x, u_Y> > 0$.

 Now let $Z = T_Z(\Phi_{n}^- )\cap \Phi_{n}^+ \in \mathcal{L}_n$ which contains both $X$ and $Y$. Let us assume that 
 $T_Z(\Bbb{R}^n) \not= \Bbb{R}^n$.  
We must have $\bar{Z}$ contains both $\bar{X}$ and $\bar{Y}$. We have  
{\small $${ \bar{Y} \cap S^{n - 1} = \{y \in S^{n -1} | 
<y, u_Y> \leq 0\} \subseteq \bar{Z} \cap S^{n - 1} = \{y \in S^{n -1} | 
<y, T_Z(e_{n + 1})> \leq 0\}.}$$}
 As in the proof of Theorem  \ref{conditions} 
we must have $H(u_y) \cap S^{n - 1} \subseteq H(T_Z(e_{n + 1})) \cap S^{n - 1}$ and hence by 
Lemma \ref{equality} we have equality. From the inclusion above we see that the half spaces 
$H_{u_Y} \cap \Bbb{R}^n$ and $H_{T_Z(e_{n + 1})} \cap \Bbb{R}^n$ are equal.
 However since we must have $\bar{X} \subseteq \bar{Z}$ and we have 
$x \in \bar{X}$ with $<x, u_Y> > 0$, we get  a contradiction. Therefore we must have $T_Z(\Bbb{R}^n) = 
\Bbb{R}^n$ and $T_Z = T'_Z \bot -1_{e_{n + 1}}$, since $Y \subseteq Z$ and we have $Y\cap \{x \in S^n| <x, e_{n + 1}> > 0 \} \not= \emptyset$. 

Let $Z_1 $ be an element of  $ \mathcal{L}_{n - 1}$ which is a least upper bound 
for $X = X \cap S^{n - 1}$ and $Y_1 = Y\cap S^{n - 1}$. Then $Z_1  = D_{Z_1} \cap \Phi_{n-1}^+$  where $D_{Z_1}$ is a maximal pointed convex cone 
of $\Bbb{R}^n$. Let $D = D_{Z_1} \cup \{x \in S^n | <x, e_{n +1}> > 0\}$ and let $Z_2 = D \cap \Phi_n^+ \in \mathcal{L}_n$. Then 
it is not difficult to see that  both $X$ and $Y$ are contained in $Z_2$. 

On the other hand if $Z = T_Z(\Phi_{n}^-) \cap \Phi_{n}^+$ is an element of  $\mathcal{L}_n$ which contains both $X$ and $Y$, from above  we have $T_Z = T'_Z \bot -1_{e_{n + 1}}$ and 
$\{x \in S^n | <x, e_{n +1}> > 0\} \subset Z$. Since $Z\cap S^{n - 1} \in \mathcal{L}_{n - 1}$ which 
contains $X$ and $Y_1$, we must have $Z_1 \subseteq Z$. Therefore $Z_2 = Z_1 \cup \{x \in S^n | <x, e_{n +1}> > 0\} \subset Z$ and $Z_2 = X\vee Y$ in $\mathcal{L}_n$. This finishes the proof. 
\end{proof}

\section{$\mathcal{L}_{n }$ is a complete lattice }
In this section, we show that $\mathcal{L}_{n }$ is a complete lattice, that is given any subset $S$ of $\mathcal{L}_{n }$,
$\bigwedge S$ and $\bigvee S$ exist in $\mathcal{L}_{n }$.
If $S$ is finite, we already have that $\bigvee S$ exists in $\mathcal{L}_{n }$, by induction and $\bigwedge S$ exists by 
duality. 
 The full result  will follow from exercise 7.5 in \cite{DnP} when we have shown that the result is true for directed sets. 

\begin{lemma} {( Exercise 7.5 \cite{DnP})} Let $P$ be an ordered set such that  the join of any two elements in $P$ exists in $P$. Let $S \subseteq P$  Let $D$ be the directed set  $D = \{\bigvee F |
\emptyset \not= F \subseteq S, F \  \mbox{finite} \}$.  Then $\bigvee S = \bigvee D$ if $\bigvee D$ exists. 
\end{lemma}
\begin{proof} If $U$ is an upper bound for $S,$ then $U$ is an upper bound for any finite subset $F$ of $S$
and hence  $\bigvee F \leq  U$. Therefore $D_1 \leq U$ for all $D_1 \in D$ and  $\bigvee D \leq  \bigvee S$
if they exist. 

On the other hand, given any $s \in S$, we have $s = s \vee s \in D$ and $s \leq \bigvee D$ if $\bigvee D$ exists. 
Therefore $\bigvee D$ is an upper bound for $S$ and $\bigvee S \leq \bigvee D$. 
\end{proof}

In order to show that $\mathcal{L}_{n }$ is a complete lattice, it remains to show that any directed set in $\mathcal{L}_{n }$
has a least upper bound. It will then follow from the above Lemma  that every subset $S$ of $\mathcal{L}_{n }$ has a least upper bound and the existence of a greatest lower bound for any set follows from duality (Lemma \ref{duality})

 Recall the definition of $\widehat{X}$ for $X \in \mathcal{L}_n$ from Definition \ref{hatconedef} and the definitions of $D_X$ and $D_e$ from 
 Definition \ref{mpccX}.
 
\begin{lemma} Let $\{X_i\}_{i \in I}$ be a directed set in $\mathcal{L}_{n}$. Let $C_X = \bigcup \widehat{X}_i$. Then 
$C_X$ is a convex cone with  $X = C_X \cap \Phi_n^+ \in \mathcal{L}_{n}$ and   $X = \bigvee \{X_i\}_{i \in I}$. 
\end{lemma}
\begin{proof} For each $X_i, i \in I$, we have $X_i = D_{X_i} \cap D_e \cap S^n$. Let $C_{X_i} = D_{X_i} \cap D_e $, then 
$C_{X_i}$ is a convex cone and by Lemma \ref{hatcone}, $\widehat{X}_i = C_{X_i}$. 
To verify  that $C_X$ is a convex cone, we need only verify that condition (3) of Definition \ref{cones} holds, since  conditions (1) and (2)  are obvious. 
 Given $u, v \in C_X$, we must  have 
$\widehat{X}_i, \widehat{X}_j$ with $u \in \widehat{X}_i$ and $v \in \widehat{X}_j$ for some $i, j \in I$. Since $\{\widehat{X}_i\}_{i \in I}$ is  a directed set, there exists
$\widehat{X}_k, k \in I$, with $\widehat{X}_i \subseteq \widehat{X}_k$ and  $\widehat{X}_i \subseteq \widehat{X}_k$. Therefore  $u, v \in \widehat{X}_k$ and since $ \widehat{X}_k$ is a convex cone,  we have $u + v \in \widehat{X}_k \subseteq C_X$. Thus $C_X$ is a 
convex cone.  

Let $X = \bigcup_{i\in I} X_i$. We have $X  = C_X \cap S^n$. 
We will  show that there is a maximal pointed convex cone $D \subset \Bbb{R}^{n + 1}$ such that $X  = D \cap \Phi^+_{n }$. 
For each $i \in I$, let $Y_i = \Phi^+_{n} \backslash X_i$. Recall that $Y_i \in \mathcal{L}_{n }$. Hence $\widehat{Y}_i$ is a convex cone for each $i \in I$ and clearly 
$C_Y = \bigcap_{i \in I} \hat{Y}_i$ is also a convex cone with $C_Y\cap C_X = \{0\}$ and $C_X \cup C_Y = D_e$. We let $D = C_X \cup (-C_Y)$. We have $\Bbb{R}^{n + 1} = C_X \cup (-C_Y) \cup C_Y \cup (-C_X)  = D \cup (-D)$. It is clear that $D \cap (-D) = \{0\}$ and that 
$D\cup (-D) = D_e \cup (- D_e) = \Bbb{R}^{n + 1}$. We show that $D$ is a maximal pointed convex cone. 

Conditions (1) and (2) 
of Definition  \ref{cones} are obvious. 
 We need only show that if $u, v, \in D$, then $u + v$ is also in $D$. Since both $C_X$ and $(-C_Y)$ each have property (3), we can assume that $u \in C_X$ and $v \in (-C_Y)$. We can also assume that $u \not= 0$ and $v \not=0$, since otherwise the result is trivial. If $u + v \not\in C_X \cup (-C_Y)$, then we must have $u + v \in (-C_X) \cup C_Y$ and 
 $u + v \in (-C_X)$ or $u + v \in   C_Y$. If $u + v = -u_1$ for some $u_1 \in C_X$, we get $u + u_1 = -v \in C_X \cap (-C_Y) = \{0\}$, implying that $v = 0$ and  giving us a contradiction. Similarly if $u + v = v_1$ where $v_1 \in C_Y$, we get a contradiction. Therefore 
 $u + v  \in C_X \cup (-C_Y) = D$ and $D$ is a maximal pointed convex set. 
 Clearly $D\cap \Phi_n^+ =  C_X \cap \Phi_n^+ = X$ and $X \in \mathcal{L}_n$.  
  Since $X = \bigcup_{i \in I}X_i$, we must have $X  = \bigvee \{X_i\}_{i \in I}$. 
\end{proof}

\end{document}